\documentclass[12pt, reqno]{amsart}
\usepackage{enumerate,amsmath,amssymb,bm,ascmac,amsthm,url}
\usepackage{tikz}
\usepackage[margin=30truemm]{geometry}
\usepackage{here}

\theoremstyle{definition}
\newtheorem{thm}{Theorem}[section]
\newtheorem{thm*}{Theorem}

\newtheorem{defi*}{Definition}
\newtheorem{lem}[thm]{Lemma}
\newtheorem{lem*}{Lemma}
\newtheorem{pro}[thm]{Proposition}
\newtheorem{pro*}{Proposition}

\newtheorem{cor}[thm]{Corollary}

\newcommand{\MC}[1]{\mathcal{#1}}
\newcommand{\MB}[1]{\mathbb{#1}}

\newcommand{\C}[2]{\begin{pmatrix} #1 \\ #2 \end{pmatrix}}

\newcommand{\marukakko}[1]{\left( #1 \right)}
\newcommand{\tyuukakko}[1]{\left \{ #1 \right \}}

\newcommand{\G}{\Gamma}

\newcommand{\NH}{\tilde{H}}

\DeclareMathOperator{\Spec}{Spec}

\DeclareMathOperator{\Tr}{Tr}

\DeclareMathOperator{\rank}{rank}

\makeatletter

\@addtoreset{equation}{section}
\makeatother

\title[
Combinatorial conditions for regular graphs to be periodic
]{
Combinatorial necessary conditions for regular graphs to induce periodic quantum walks
}

\author[S. Kubota]{Sho Kubota$^1$}
\thanks{$^1$S.K. is supported by JSPS KAKENHI (Grant No. 20J01175).}
\address{Department of Information Systems,
Faculty of Information Science and Technology,
Osaka Institute of Technology,
Osaka 573-0196, Japan.}
\email{sho.kubota@oit.ac.jp}

\date{}

\begin{document}
\maketitle
\begin{abstract}
We derive combinatorial necessary conditions
for discrete-time quantum walks defined by regular mixed graphs to be periodic.
One useful necessary condition is that
if a $k$-regular mixed graph with $n$ vertices is periodic, then $2n/k$ must be an integer.
As an application of this work,
we determine periodicity of mixed complete graphs
and mixed graphs with a prime number of vertices.
Furthermore,
we study periodicity of mixed strongly regular graphs and several classes of mixed distance-regular graphs,
and extend existing results to mixed graphs.
\end{abstract}

\section{Introduction} \label{0505-1}

Let $\G =(V, E)$ be an undirected graph with the vertex set $V$ and the edge set $E$.
Define $\MC{A} = \MC{A}(\G)=\{ (x, y), (y, x) \mid \{x, y\} \in E \}$,
which is the set of the symmetric arcs of $\G$.
The origin and terminus of $a=(x, y) \in \MC{A}$ are denoted by $o(a), t(a)$, respectively.
We use $a^{-1}$ to denote the reverse arc of $a$.
Define the matrix $U = U(\G) \in \MB{C}^{\MC{A} \times \MC{A}}$ by
\[ U_{a,b} = \frac{2}{\deg_{\G} t(b)} \delta_{o(a), t(b)} - \delta_{a, b^{-1}}, \]
where $\delta_{x,y}$ is the Kronecker delta symbol.
Consider the 3-cycle $C_3$ for example.
Then we have
\[
U=U(C_3) = \begin{bmatrix}
0&0&1&0&0&0 \\
1&0&0&0&0&0 \\
0&1&0&0&0&0 \\
0&0&0&0&1&0 \\
0&0&0&0&0&1 \\
0&0&0&1&0&0 
\end{bmatrix}.
\]
In this case,
the matrix $U$ is a block diagonal matrix whose diagonal blocks are lined with permutation matrices.
Thus we see that $U^3 = I$,
where $I$ is the identity matrix.
We consider the question of whether there exists $\tau \in \MB{N}$
such that $U(\G)^{\tau} = I$ for various graphs.
This is nothing but the question of periodicity of discrete-time quantum walks.
An application of the periodicity can be seen in quantum cryptography \cite{PB}.

\subsection{Previous works on the periodicity of quantum walks}

The topic discussed in this paper is periodicity of discrete-time quantum walks.
See Table~\ref{80}.
Many studies in the Grover walk model are carried out.
Studies of the periodicity in other models are summarized in Table~\ref{0430-0}.
Very recently, the periodicity of mixed graphs has been studied,
and the periodicity of the mixed paths and the mixed cycles was completely determined \cite{KSeYa}.
State transfer, which is a more generalized problem of the periodicity, is discussed in \cite{CZ, Z}.

\begin{table}
  \centering
  \begin{tabular}{|c|c|}
\hline
Graphs & Ref. \\
\hline
\hline
Complete graphs, complete bipartite graphs, SRGs & \cite{HKSS2017} \\ \hline
Generalized Bethe trees & \cite{KSTY2018} \\ \hline
Hamming graphs, Johnson graphs & \cite{Y2019} \\ \hline
Cycle (3-state) & \cite{KKKS} \\ \hline
Complete graphs with self loops & \cite{IMT} \\ \hline
\end{tabular}
 \caption{Previous works on the periodicity of Grover walks on undirected graphs} \label{80}
\end{table}

\begin{table}
  \centering
  \begin{tabular}{|c|c|}
\hline
Models & Ref. \\
\hline
\hline
Quantum walks on $C_n$ with the most general $2 \times 2$ unitary coins & \cite{D} \\ \hline
Hadamard walks & \cite{KoST} \\ \hline
Szegedy walks induced by lazy random walks &\cite{HKSS2017} \\ \hline
Szegedy walks induced by non-reversible random walks on $C_n$ &\cite{HKSS2017} \\ \hline 
Fourier walks & \cite{S, KKKS} \\ \hline
Staggered walks & \cite{KSTY2019} \\ \hline
Lively quantum walks with generalized Grover coins & \cite{SMA} \\ \hline
\end{tabular}
 \caption{Previous works on the periodicity for other models} \label{0430-0}
\end{table}

Studies of periodicity in continuous-time quantum walks are in \cite{G2011, SSS} for example.
For continuous-time quantum walk, there is a lot of research into state transfer.
We refer to Godsil's survey \cite{G2012}.
Recent studies on state transfer are in \cite{BMW, LW, LLZZ, MDDT, TYC, WL}.

\subsection{Main Results}
The main goal of this study is to determine the periodicity of the mixed complete graphs.
In the case of undirected graphs,
there is only one complete graph for $n$,
whereas,
in the case of mixed complete graphs,
there are $3^{\frac{n(n-1)}{2}}$ distinct mixed complete graphs on any vertex set $V$ of size $n$.
In this paper,
we show that mixed complete graphs are not periodic for $n \geq 4$,
as is the case for undirected graphs.
See later sections for more detailed terms and definitions.
The following is one of the main results.

\begin{thm}
{\it
Let $G$ be a mixed complete graph equipped with an $\eta$-function $\theta$.
If $n \geq 4$, then $G$ is not periodic.
}
\end{thm}

If mixed graphs are regular,
we can find explicit relations between the characteristic polynomials of $\eta$-Hermitian adjacency matrices
and that of the time evolution matrices.
On the other hand,
coefficients of the characteristic polynomials of $\eta$-Hermitian adjacency matrices have combinatorial implications.
In addition to this,
if a mixed graph is periodic,
then all the eigenvalues of the time evolution matrix must be algebraic integers.
From the above considerations,
we can obtain the following combinatorial conditions,
which is the second main result of this study.

\begin{thm} \label{0902-1}
{\it
Let $G$ be a mixed graph on $n$ vertices equipped with an $\eta$-function $\theta$.
Assume that $G$ is $k$-regular and periodic.
Then,
\begin{enumerate}[(i)]
\item We have $\frac{2n}{k} \in \MB{Z}$; and
\item If $G$ is an undirected graph with $t$ triangles,
then we have $\frac{16t}{k^3} \in \MB{Z}$.
\end{enumerate}
}
\end{thm}

The condition in the above theorem leads to an interesting fact
when the number of vertices in a regular mixed graph is prime.
The following is the third main result.

\begin{thm}
{\it
Let $\eta \in \MB{Q}\pi$,
and let $G$ be a $k$-regular weakly connected mixed graph with $p$ vertices,
where $p$ is a prime number.
Then the mixed graph $G$ is periodic if and only if $G^{\pm}$ is isomorphic to either
the cycle graph $C_p$ or the path graph $P_2$.
}
\end{thm}

In addition,
we study the periodicity of mixed strongly regular graphs and several classes of mixed distance-regular graphs,
and extend the results in \cite{HKSS2017} and \cite{Y2019} to mixed graphs.
The results indicate that there could potentially be $\overline{L(K_6)}$, Hamming graphs,
or Johnson graphs whose mixed version is periodic and the undirected version is not.

This paper is organized as follows.
Section~\ref{0504-2} is mainly preparation.
We define several matrices that appear in this paper.
In Section~\ref{0424-01},
we discuss relationships between the coefficients of the characteristic polynomials of the matrices.
In Section~\ref{0504-4}, we discuss conditions on the coefficients of the characteristic polynomials of the time evolution matrices for mixed graphs to be periodic.
If mixed graphs are periodic, the eigenvalues of the time evolution matrices must be algebraic integers.
We focus on this and study which ring the coefficients belong to.
In Section~\ref{0504-5},
we classify the periodic mixed complete graphs as an application of the previous sections.
The proof itself is straightforward.
In Section~\ref{0902-6},
we study the periodicity of mixed strongly regular graphs.
In particular, we can reproduce the results of Higuchi et al \cite{HKSS2017}
by applying Theorem~\ref{0902-1} to undirected strongly regular graphs.
In Section~\ref{0902-7}, we consider the periodicity of mixed Hamming graphs.
In Section~\ref{0902-8}, we investigate the periodicity of mixed sesquilinear forms graphs.
Using only Theorem~\ref{0902-1}~(i), we derive that these graphs are not periodic.
In Section~\ref{0902-9}, we consider the periodicity of mixed Johnson graphs.

\section{Preliminaries} \label{0504-2}


A {\it mixed graph} $G$ consists of a finite set $V(G)$ of vertices
together with a subset $\MC{A}(G) \subset V(G) \times V(G) \setminus \{ (x,x) \mid x \in V \}$
of ordered pairs called {\it arcs}.
Let $G$ be a mixed graph.
Define $\MC{A}^{-1}(G) = \{ a^{-1} \mid a \in \MC{A}(G) \}$
and $\MC{A}^{\pm}(G) = \MC{A}(G) \cup \MC{A}^{-1}(G)$.
If there is no danger of confusion, we write $\MC{A}(G)$ as $\MC{A}$ simply.
The graph $G^{\pm} = (V(G), \MC{A}^{\pm})$ is
so-called the {\it underlying graph} of a mixed graph $G$,
and this is regarded as an undirected graph depending on context.
On the other hand,
we equate an undirected graph with a mixed graph
by considering undirected edges $\{x, y\}$ as bidirectional arcs $(x,y), (y,x)$.
For a vertex $x \in V(G)$,
define $\deg_{G} x = \deg_{G^{\pm}} x$.
A mixed graph $G$ is $k$-regular if $\deg_{G}x = k$ for any vertex $x \in V(G)$.
Throughout this paper, we assume that mixed graphs are weakly connected,
i.e.,
we assume that $G^{\pm}$ is connected.

\subsection{$\eta$-Hermitian adjacency matrices}

Let $G = (V, \MC{A})$ be a mixed graph.
For $\eta \in [0, 2\pi)$, 
the {\it $\eta$-Hermitian adjacency matrix} $H_{\eta} = H_{\eta}(G) \in \MB{C}^{V \times V}$ is defined by
\[
(H_{\eta})_{x,y} = \begin{cases}
1 \qquad &\text{if $(x,y) \in \MC{A} \cap \MC{A}^{-1}$,} \\
e^{\eta i} \qquad &\text{if $(x,y) \in \MC{A} \setminus \MC{A}^{-1}$,} \\
e^{-\eta i} \qquad &\text{if $(x,y) \in \MC{A}^{-1} \setminus \MC{A}$,} \\
0 \qquad &\text{otherwise.}
\end{cases}
\]
When $\eta = \frac{\pi}{2}$,
the matrix $H_{\frac{\pi}{2}}$ is nothing but the Hermitian adjacency matrix.
This is introduced by Guo--Mohar \cite{GM} and Li--Liu \cite{LL}, independently.
When $\eta = \frac{\pi}{3}$,
the matrix $H_{\frac{\pi}{3}}$ is called the Hermitian adjacency matrix of the second kind.
This is introduced by Mohar \cite{M}.
We refer to \cite{AAS, GS, HBHZ, HF, LY, YWY} as recent studies on Hermitian adjacency matrices.
Note that $H_{\eta}(G^{\pm})$ coincides with the adjacency matrix of $G^{\pm}$ as an undirected graph.
We also refer to gain graphs \cite{GH, LWH, MKS, R, SK, WvD, XZWT},
which are a more generalized version of mixed graphs.

Define the {\it degree matrix} $D = D(G) \in \MB{C}^{V \times V}$ by $D_{x,y} = (\deg_{G}x)\delta_{x,y}$
for vertices $x,y \in V(G)$.
For $\eta \in [0, 2\pi)$,
the {\it normalized $\eta$-Hermitian adjacency matrix} $\NH_{\eta}$ is defined by
\[ \NH_{\eta} = D^{-\frac12} H_{\eta} D^{-\frac12}. \]
Note that if a mixed graph $G$ is $k$-regular, we have $\NH_{\eta} = \frac{1}{k} H_{\eta}$.

Let $G$ be a mixed graph.
The list of the eigenvalues of $H_{\eta}(G)$ together with their multiplicities,
denoted by $\Spec(H_{\eta}(G))$, is called {\it $H_{\eta}$-spectrum} of $G$.

\subsection{Quantum walks defined by mixed graphs}
Let $\eta \in [0, 2\pi)$,
and let $G = (V, \MC{A})$ be a mixed graph.
The {\it $\eta$-function} $\theta : \MC{A}^{\pm} \to \MB{R}$ of a mixed graph $G$ is defined by
\[
\theta(a) = \begin{cases}
\eta \qquad &\text{if $a \in \MC{A} \setminus \MC{A}^{-1}$,} \\
-\eta \qquad &\text{if $a \in \MC{A}^{-1} \setminus \MC{A}$,} \\
0 \qquad &\text{if $a \in \MC{A} \cap \MC{A}^{-1}$.}
\end{cases}
\]
Note that $\theta(a^{-1}) = -\theta(a)$ for any $a \in \MC{A}^{\pm}$.

In \cite{KST},
the authors provided a quantum walk defined by a mixed graph.
Let $G = (V, \MC{A})$ be a mixed graph equipped with an $\eta$-function $\theta$.
We define several matrices related to quantum walks.
The {\it boundary matrix} $K = K(G) \in \MB{C}^{V \times \MC{A}^{\pm}}$ is defined by
\[ K_{x,a} = \frac{1}{\sqrt{\deg_G x}} \delta_{x,t(a)}. \]
The {\it coin matrix} $C = C(G) \in \MB{C}^{\MC{A}^{\pm} \times \MC{A}^{\pm}}$ is defined by
$C = 2K^*K-I$.
The {\it shift matrix} $S_{\theta} = S_{\theta}(G) \in \MB{C}^{\MC{A}^{\pm} \times \MC{A}^{\pm}}$
is defined by $(S_{\theta})_{ab} = e^{\theta(b)i}\delta_{a,b^{-1}}$.
Define the {\it time evolution matrix} $U_{\theta} = U_{\theta}(G) \in \MB{C}^{\MC{A}^{\pm} \times \MC{A}^{\pm}}$ by $U_{\theta} = S_{\theta} C$.
The entries of $U_{\theta}$ is calculated in \cite{KSeYa}.

\begin{lem}[Lemma~5.1 in \cite{KSeYa}] \label{22}
{\it
Let $G = (V, \MC{A})$ be a mixed graph equipped with an $\eta$-function $\theta$.
We have
\[ (U_{\theta})_{a,b} = e^{-\theta(a) i} \marukakko{
  \frac{2}{\deg_{G} t(b)} \delta_{o(a), t(b)} - \delta_{a, b^{-1}}
  } \]
for any $a,b \in \MC{A}^{\pm}$.
}
\end{lem}

If a mixed graph $G$ is an undirected graph,
i.e., $G^{\pm} = G$,
then $\theta(a) = 0$ for any $a \in \MC{A}$.
Thus, in this case,
the matrix $U_{\theta}$ is the time evolution matrix of the conventional Grover walk on the graph,
given as the matrix $U$ in Section~\ref{0505-1}.

The matrices used in our quantum walks are summarized in Table~\ref{1000},
where $G = (V, \MC{A})$ is a mixed graph equipped with an $\eta$-function $\theta$.

\begin{table}[H]
  \centering
  \begin{tabular}{|c|c|c|c|}
\hline
Notation & Name & Indices of rows and columns & Definition \\
\hline
\hline
$K$ & Boundary & $V \times \MC{A}^{\pm}$ & $K_{x,a} =  \frac{1}{ \sqrt{\deg x} } \delta_{x, t(a)}$ \\
\hline
$C$ & Coin & $\MC{A}^{\pm} \times \MC{A}^{\pm}$ &$ C = 2K^*K - I$ \\
\hline
$S_{\theta}$ & Shift & $\MC{A}^{\pm} \times \MC{A}^{\pm}$ & $(S_{\theta})_{ab} = e^{\theta(b)i}\delta_{a,b^{-1}}$ \\
\hline
$U_{\theta}$ & Time evolution & $\MC{A}^{\pm} \times \MC{A}^{\pm}$  & $U_{\theta} = S_{\theta} C$ \\
\hline
\end{tabular}
 \caption{The matrices used in our quantum walk} \label{1000}
\end{table}

Let $U_{\theta}$ be a time evolution matrix of a mixed graph $G$ equipped with an $\eta$-function $\theta$.
We say that $G$ is {\it periodic} if there exists $\tau \in \MB{N}$ such that $U_{\theta}^{\tau} = I$.
In \cite{KSeYa},
the authors provide a necessary and sufficient condition for a mixed graph to be periodic.

\begin{lem}[Lemma~5.3 in \cite{KSeYa}] \label{13}
{\it
Let $U_{\theta}$ be a time evolution matrix of a mixed graph $G$ equipped with an $\eta$-function $\theta$.
Then, we have
\begin{equation} \label{70}
\{ \tau \in \MB{N} \mid U_{\theta}^{\tau} = I \} =
\{ \tau \in \MB{N} \mid \lambda^{\tau} = 1 \text{ \emph{for any} $\lambda \in \Spec(U_{\theta})$} \}.
\end{equation}
In particular, $G$ is periodic if and only if
there exists $\tau \in \MB{N}$ such that $\lambda^{\tau} = 1$
for any eigenvalue $\lambda$ of $U_{\theta}$.
}
\end{lem}

\section{Relationships between coefficients of characteristic polynomials} \label{0424-01}

It is well-studied relationships between the eigenvalues of time evolution matrices
of quantum and random walks \cite{HKSS2014, KoKKS, KSY, MOS, SS}.
In our setting,
$U_{\theta}$ is the time evolution matrix of the quantum walk and
$\NH_{\eta}$ is a matrix corresponding to a random walk.
We cite the result in \cite{KoKKS}.

\begin{thm}[Corollary~3 in \cite{KoKKS}]
{\it
Let $U_{\theta}$ be a time evolution matrix of a mixed graph $G$ equipped with an $\eta$-function $\theta$,
and let $\NH_{\eta} = \NH_{\eta}(G)$. 
Then, we have
\[ \det(\lambda I_{2m} - U_{\theta})
= (\lambda^2 - 1)^{m-n} \det( (\lambda^2 + 1)I_n - 2\lambda \NH_{\eta}), \]
where $n$ is the number of vertices of $G^{\pm}$,
and $m$ is the number of edges of $G^{\pm}$.
}
\end{thm}

In the above equality, 
we call $\det( (\lambda^2 + 1)I_n - 2\lambda \NH_{\eta})$ the {\it inherited factor} of $U_{\theta}$.
Since the roots of $(\lambda^2 - 1)^{m-n}$ are $\pm 1$,
the inherited factor determines whether the graph is periodic or not.
Let $g(x) = \det (xI - \NH_{\eta})$.
Then the inherited factor can be written as
\[
\det( (\lambda^2 + 1)I_n - 2\lambda \NH_{\eta})
= (2\lambda)^n \det \marukakko{ \frac{\lambda + \lambda^{-1}}{2} I_n - \NH_{\eta} }
= (2\lambda)^n g \marukakko{ \frac{\lambda + \lambda^{-1}}{2} }.
\]

Fix an positive integer $n$.
Define
$\MC{I}_j = \{ (i,l) \in \MB{Z} \times \MB{Z} \mid 0 \leq l \leq i \leq n, n+2l - i = j \}$
for $j \in \{ 0, 1, \dots, 2n \}$.
We provide $\MC{I}_{2n-3}$ in Figure~\ref{0429-0} as an example.
The black dots in Figure~\ref{0429-0} represent
the set $\{ (i,l) \in \MB{Z} \times \MB{Z} \mid 0 \leq l \leq i \leq n \}$.
The set $\MC{I}_{2n-3}$ is the intersection of
these black dots and the blue line $l = \frac{i}{2} + \frac{n-3}{2}$.
The set $\MC{I}_j$ will be discussed later.

\begin{figure}[h]
\begin{center}
\begin{tikzpicture}
[scale = 0.7, line width = 0.7pt,
v/.style = {circle, fill = black, inner sep = 0.8mm}]
\draw[->] (-1, 0) -- (8,0);
\draw[->] (0, -1) -- (0,8);
\draw (8.5, 0) node {$i$};
\draw (0, 8.5) node {$l$};
\draw[line width = 1pt] (-0.5, -0.5) -- (7.5, 7.5);
\draw (7.5, 8.3) node {$l = i$};
\draw (-0.15, 1) -- (0.15, 1); \draw (-0.5, 1) node {$1$};
\draw (-0.15, 2) -- (0.15, 2); \draw (-0.5, 2) node {$2$};
\draw (-0.15, 5) -- (0.15, 5); \draw (-1, 5) node {$n-2$};
\draw (-0.15, 6) -- (0.15, 6); \draw (-1, 6) node {$n-1$};
\draw (-0.15, 7) -- (0.15, 7); \draw (-0.5, 7) node {$n$};
\draw (1, -0.15) -- (1,0.15); \draw (1, -0.5) node {$1$};
\draw (2, -0.15) -- (2,0.15); \draw (2, -0.5) node {$2$};
\draw (6, -0.15) -- (6,0.15); \draw (6, -0.5) node {$n-1$};
\draw (7, -0.15) -- (7,0.15); \draw (7, -0.5) node {$n$};
\node[v] at (0, 0) {};
\node[v] at (1, 0) {};
\node[v] at (1, 1) {};
\node[v] at (2, 0) {};
\node[v] at (2, 1) {};
\node[v] at (2, 2) {};
\node[v] at (6, 0) {};
\node[v] at (6, 1) {};
\node[v] at (6, 2) {};
\node[v] at (6, 5) {};
\node[v] at (6, 6) {};
\node[v] at (7, 0) {};
\node[v] at (7, 1) {};
\node[v] at (7, 2) {};
\node[v] at (7, 5) {};
\node[v] at (7, 6) {};
\node[v] at (7, 7) {};
\draw (4, 2) node {$\cdots$};
\draw (4, 1) node {$\cdots$};
\draw (4, 0.25) node {$\cdots$};
\draw (6, 3.5) node {$\vdots$};
\draw (7, 3.5) node {$\vdots$};
%
\draw[line width = 1pt, blue] (1.6, 2.8) -- (7.4, 5.7);
\draw[blue] (9, 6.2) node {$l = \frac{i}{2} + \frac{n-3}{2}$};
\end{tikzpicture}
\end{center}
\caption{The set $\MC{I}_{2n-3}$} \label{0429-0}
\end{figure}


Let $U_{\theta}$ be a time evolution matrix of a mixed graph $G = (V, \MC{A})$ equipped with an $\eta$-function $\theta$,
and let
\begin{align*}
f(x) &= \det (xI - H_{\eta}(G)) = \sum_{i=0}^{n} c_i x^i, \\
g(x) &= \det (xI - \NH_{\eta}(G)) = \sum_{i=0}^{n} d_i x^i, \\
\Psi(x) &= (2x)^{n} g \marukakko{\frac{x+x^{-1}}{2}} = \sum_{i=0}^{2n} \alpha_i x^i,
\end{align*}
where $n$ is the number of vertices of $G$.
From now on,
these notations will be used unless otherwise noted.
Note that $c_{n-1} = d_{n-1} = 0$ since $\Tr H_{\eta}(G) = \Tr \NH_{\eta}(G) = 0$.
First, we explore relationships between coefficients of $f, g,$ and $\Psi$.

\begin{lem} \label{0415-01}
{\it
With the above notation,
we have
\[ \alpha_j = \sum_{(i,l) \in \MC{I}_j} 2^{n-i} d_i \C{i}{l}. \]
} 
\end{lem}

\begin{proof}
Indeed,
\begin{align*}
\sum_{j=0}^{2n} \alpha_j x^j
&= (2x)^{n} g \marukakko{\frac{x+x^{-1}}{2}} \\
&= (2x)^{n} \marukakko{ \sum_{i=0}^n d_i \marukakko{ \frac{x+x^{-1}}{2} }^i } \\
&= (2x)^{n} \marukakko{ \sum_{i=0}^n \frac{d_i}{2^i} \sum_{l=0}^i \C{i}{l} x^l (x^{-1})^{i-l} } \\
&= (2x)^{n} \sum_{i=0}^n \frac{d_i}{2^i} \sum_{l=0}^i \C{i}{l} x^{2l-i} \\
&= \sum_{i=0}^n \sum_{l=0}^i 2^{n-i} d_i \C{i}{l} x^{n+2l-i} \\
&= \sum_{j=0}^{2n} \sum_{(i,l) \in \MC{I}_j} 2^{n-i} d_i \C{i}{l} x^j.
\end{align*}
\end{proof}

\begin{cor} \label{0415-02}
{\it
We have
  \begin{enumerate}[(i)]
  \item $\alpha_{2n} = 1$;
  \item $\alpha_{2n-1} = 0$;
  \item $\alpha_{2n-2} = n + 4 d_{n-2}$; and
  \item $\alpha_{2n-3} = 8 d_{n-3}$.
  \end{enumerate}
In particular,
the inherited factor $\Psi(x)$ of $U_{\theta}$ is a monic polynomial.
}
\end{cor}
\begin{proof}
First, we claim that
\[ \MC{I}_j =
\begin{cases}
\{ (n,n) \} \quad &\text{if $j = 2n$,} \\
\{ (n-1, n-1) \} \quad &\text{if $j = 2n-1$,} \\
\{ (n, n-1), (n-2, n-2) \} \quad &\text{if $j = 2n-2$,} \\
\{ (n-1, n-2), (n-3, n-3) \} \quad &\text{if $j = 2n-3$.}
\end{cases}
\]
We show only the case $j = 2n-3$.
The other cases can be shown in the same manner.
The set $\MC{I}_{2n-3}$ is illustrated in Figure~\ref{0429-0}.
Let the line $L$ be $l=\frac{i}{2} + \frac{n-3}{2}$.
The line $L$ has the slope $1/2$,
so if $(p, q) \in \MC{I}_{2n-3}$ and $(p-2, q-1) \in \{ (i,l) \in \MB{Z} \times \MB{Z} \mid 0 \leq l \leq i \leq n \}$,
then $(p-2, q-1) \in \MC{I}_{2n-3}$.
Repeating this argument, we have $\MC{I}_{2n-3} = \{ (n-1,n-2), (n-3,n-3) \}$.
By Lemma~\ref{0415-01}, the statement is obtained.
\end{proof}

\begin{lem} \label{0415-03}
{\it
If $G$ is $k$-regular,
then we have
\[ g(x) = \frac{1}{k^n} f(kx). \]
In particular, $d_i = \frac{c_i}{k^{n-i}}$ for $i \in \{0,1,\dots, n \}$.
}
\end{lem}

\begin{proof}
Since $\NH_{\eta}(G) = \frac{1}{k}H_{\eta}(G)$,
we have
\begin{align*}
g(x) &= \det(xI - \NH_{\eta}(G)) \\
&= \det(xI - \frac{1}{k}H_{\eta}(G)) \\
&= \frac{1}{k^n} \det(kx I - H_{\eta}(G)) = \frac{1}{k^n} f(kx).
\end{align*}
Therefore,
\[
\sum_{i=0}^n d_i x^i = g(x) = \frac{1}{k^n} f(kx) = \frac{1}{k^n} \sum_{i=0}^n c_i (kx)^i
 = \sum_{i=0}^n \frac{c_i}{k^{n-i}}x^i.
\]
We have $d_i = \frac{c_i}{k^{n-i}}$ for $i \in \{0,1,\dots, n \}$.
\end{proof}

The coefficients of $H_{\eta}$ have combinatorial implications.
See Corollary~3.1 in \cite{MKS} and Theorem~3.3 in \cite{GM} for example.
In this paper, we mainly focus on $c_{n-2}$ and $c_{n-3}$.
The following lemma has been shown by Guo and Mohar \cite{GM} for the case $\eta = \frac{\pi}{2}$.
The proof is valid for general $\eta$.

\begin{lem} \label{0415-04}
{\it
We have $c_{n-2} = -|E(G^{\pm})|$.
}
\end{lem}

\begin{proof}
Let $H_{\eta} = H_{\eta}(G)$.
Since $H_{\eta}$ is Hermitian,
\begin{equation} \label{0413-1}
(H_{\eta}^2)_{x,x} = \sum_{z \in V} (H_{\eta})_{x,z} (H_{\eta})_{z,x}
= \sum_{z \in V} |(H_{\eta})_{x,z}|^2 
= \deg_{G^{\pm}} x
\end{equation}
for any $x \in V$.
Let $\lambda_1, \dots, \lambda_n$ be the eigenvalues of $H_{\eta}$.
Remark that
\begin{equation} \label{0413-2}
c_{n-2} = \sum_{1 \leq i < j \leq n} \lambda_i \lambda_j.
\end{equation}
We have
\begin{align*}
0 &= \marukakko{ \sum_{i = 1}^n \lambda_i }^2 \\
&= \sum_{i = 1}^n \lambda_i^2 + 2 \sum_{1 \leq i < j \leq n} \lambda_i \lambda_j \\
&= \Tr(H_{\eta}^2) + 2c_{n-2} \tag{by (\ref{0413-2})}\\
&= \sum_{x \in V} \deg_{G^{\pm}} x + 2c_{n-2} \tag{by (\ref{0413-1})} \\
&= 2|E(G^{\pm})| + 2c_{n-2}. \tag{by Handshaking Lemma}
\end{align*}
Therefore, we have $c_{n-2} = -|E(G^{\pm})|$.
\end{proof}

\begin{pro} \label{0416-1}
{\it
If $G$ is $k$-regular,
then we have $\alpha_{2n-2} = n-\frac{2n}{k}$.
}
\end{pro}

\begin{proof}
Indeed,
\begin{align*}
\alpha_{2n-2}
&= n + 4 d_{n-2} \tag{by Corollary~\ref{0415-02}~(iii)} \\
&= n + \frac{4 c_{n-2}}{k^2} \tag{by Lemma~\ref{0415-03}} \\
&= n - \frac{4 |E(G^{\pm})|}{k^2} \tag{by Lemma~\ref{0415-04}} \\
&= n - \frac{2n}{k}. \tag{by Handshaking Lemma}
\end{align*}
\end{proof}

It is important to note that the above argument is independent of $\eta$.

\section{Rings to which coefficients belong} \label{0504-4}

In this section,
we investigate the ring containing the coefficients of the inherited factor of $U_{\theta}$ for periodic mixed graphs.
Let $\alpha \in \MB{C}$.
Recall that $\MB{Z}[\alpha] = \{ p(\alpha) \mid p(x) \in \MB{Z}[x] \}$,
and $\MB{Q}(\alpha)$ is the smallest subfield of $\MB{C}$ containing $\MB{Q}$ and $\alpha$.
Explicitly,
it is described as
\[ \MB{Q}(\alpha) = \tyuukakko{ \frac{p(\alpha)}{q(\alpha)} \, \middle | \, p(x), q(x) \in \MB{Q}[x],
q(\alpha) \neq 0 }. \]
We use the same notation as in Section~\ref{0424-01}.

\begin{lem} \label{0415-07}
{\it
We have $g(x) \in \MB{Q}(e^{i \eta})[x]$.
}
\end{lem}
\begin{proof}
We have $(D^{-1}H_{\eta}(G))_{x,y} \in \MB{Q}(e^{i \eta})$ for any $x, y \in V$.
Thus,
$g(x) =  \det(xI - \NH_{\eta}(G)) = \det(xI - D^{-1}H_{\eta}(G)) \in \MB{Q}(e^{i \eta})[x]$.
\end{proof}

Let $\alpha \in \MB{C}$.
We say that $\alpha$ is an {\it algebraic integer}
if there exists a monic polynomial $p(x) \in \MB{Z}[x]$ such that $p(\alpha) = 0$.
Let $\Omega$ denote the set of algebraic integers.
We recall some basic facts on the algebraic integers.
Let $\zeta_m = e^{\frac{2\pi}{m}i}$.

\begin{pro}[\cite{J, L}] \label{0424-02}
{\it
We have the following.
\begin{enumerate}[(i)]
\item $\Omega$ is a subring of $\MB{C}$;
\item $\Omega \cap \MB{Q} = \MB{Z}$; and
\item $\Omega \cap \MB{Q}(\zeta_m) = \MB{Z}[\zeta_m]$.
\end{enumerate}
}
\end{pro}

\begin{pro} \label{0416-2}
{\it
If $G$ is periodic, then the following hold.
  \begin{enumerate}[(i)]
  \item $\Psi(x) \in \Omega[x]$;
  \item If $\eta \in \MB{Q}\pi$,
  there exists a positive integer $m$ such that $\Psi(x) \in (\MB{Z}[\zeta_m])[x]$; and
  \item If $\eta \in \{0, \pi/2 \}$, then $\Psi(x) \in \MB{Z}[x]$.
  \end{enumerate}
}
\end{pro}

\begin{proof}
(i) By Lemma~\ref{13},
there exists $\tau \in \MB{N}$ such that $\lambda^{\tau} = 1$
for any eigenvalue $\lambda$ of $U_{\theta}$.
Any eigenvalue $\lambda$ of $U_{\theta}$
is a root of the monic polynomial $p(x) = x^{\tau} - 1 \in \MB{Z}[x]$,
so $\lambda \in \Omega$.
Let $\lambda_1, \lambda_2, \dots, \lambda_{2n}$ be roots of $\Psi(x)$.
Since $\Psi(x)$ is monic and $\Omega$ is a ring,
\[ \Psi(x) = (x-\lambda_1)(x-\lambda_2) \cdots (x - \lambda_{2n}) \in \Omega[x]. \]

(ii) 
Let $\eta = \frac{2l}{m}\pi$, where $l$ and $m$ are coprime.
From Lemma~\ref{0415-07},
we have $g(x) \in \MB{Q}(\zeta_m)[x]$.
Lemma~\ref{0415-01} implies $\Psi(x) \in \MB{Q}(\zeta_m)[x]$.
From (i) and Proposition~\ref{0424-02}~(iii),
we have $\Psi(x) \in (\Omega \cap \MB{Q}(\zeta_m))[x] = (\MB{Z}[\zeta_m])[x]$.

(iii) If $\eta = 0$, then clearly $g(x) \in \MB{Q}[x]$.
If $\eta = \pi/2$, then it also holds that $g(x) \in \MB{Q}[x]$.
Indeed,
the eigenvalues of $\NH_{\frac{\pi}{2}}$ are real numbers,
so $g(x) \in \MB{R}[x]$.
By Lemma~\ref{0415-07},
we have $g(x) \in \MB{Q}(e^{\frac{\pi}{2}i})[x] = \MB{Q}(i)[x]$.
Therefore,
$g(x) \in (\MB{R} \cap \MB{Q}(i))[x] = \MB{Q}[x]$.
Now, let $\eta \in \{0, \pi/2 \}$.
Lemma~\ref{0415-01} implies 
$\Psi(x) \in \MB{Q}[x]$.
From (i), we have $\Psi(x) \in (\Omega \cap \MB{Q})[x] = \MB{Z}[x]$.
\end{proof}

We note that the claim in Proposition~\ref{0416-2}~(iii) for $\eta = 0$ has been substantially clarified
by Higuchi et al \cite{HKSS2017}.
We here derived a more general statement in a different way to theirs.
The above proposition and the coefficient $\alpha_{2n-2}$ of the inherited factor
yield a very simple necessary criterion for mixed graphs to be periodic,
which is one of the main results.

\begin{cor} \label{0416-3}
{\it
Let $G$ be a mixed graph on $n$ vertices equipped with an $\eta$-function $\theta$.
If $G$ is $k$-regular and periodic,
then we have $\frac{2n}{k} \in \MB{Z}$.
}
\end{cor}

\begin{proof}
Let $\Psi(x)$ be the inherited factor of the time evolution matrix $U_{\theta}$ of $G$.
Display as $\Psi(x) = \sum_{i=0}^{2n} \alpha_i x^i$.
By Proposition~\ref{0416-1}, $\alpha_{2n-2} = n - \frac{2n}{k} \in \MB{Q}$.
By Proposition~\ref{0416-2}~(i), $\alpha_{2n-2} \in \Omega$.
Thus, $n - \frac{2n}{k} = \alpha_{2n-2} \in \Omega \cap \MB{Q} = \MB{Z}$.
We have $\frac{2n}{k} \in \MB{Z}$.
\end{proof}

\subsection{Undirected graphs}

We give a more developed consideration to undirected graphs.
Note that if $\G$ is an undirected graph,
then $H_{\eta}(\G)$ is the ordinary adjacency matrix for any $\eta$.
In addition, for an arbitrary $\eta$-function $\theta$,
we have $U_{\theta}(\G) = U_{\theta_0}(\G)$,
where $\theta_0$ is the $0$-function.
We continue to use the same notation as in Section~\ref{0424-01}.

\begin{lem} \label{0426-01}
{\it
Let $\G$ be an undirected graph,
and let $t$ be the number of triangles in $\G$.
Then we have $c_{n-3} = -2t$.
}
\end{lem}

\begin{proof}
Let $A$ be the adjacency matrix of $\G$.
Let $\lambda_1, \lambda_2, \dots, \lambda_n$ be the eigenvalues of $A$.
Remark that
\[
c_{n-3} = - \sum_{1 \leq i < j < k \leq n} \lambda_i \lambda_j \lambda_k.
\]
Thus,
\begin{align*}
0 &= \marukakko{ \sum_{i=1}^n \lambda_i }^3 \\
&= \sum_{i=1}^n \lambda_i^3 + 3\sum_{\substack{ i,j \\ i \neq j }} \lambda_i^2 \lambda_j + 6 \sum_{1 \leq i < j < k \leq n} \lambda_i \lambda_j \lambda_k \\
&= \Tr{A^3} + 3 \sum_{i=1}^n \lambda_i^2 \sum_{\substack{ j=1\\ j \neq i }}^n \lambda_j -6c_{n-3} \\
&= 6t + 3 \sum_{i=1}^n \lambda_i^2 (-\lambda_i) -6c_{n-3} \\
&= 6t - 3 \sum_{i=1}^n \lambda_i^3 -6c_{n-3} \\
&= 6t - 18t -6c_{n-3} \\
&= -12t - 6c_{n-3}.
\end{align*}
Therefore, we have $c_{n-3} = -2t$.
\end{proof}

\begin{pro} \label{0806-2}
{\it
Let $\G$ be an undirected graph with $t$ triangles.
If $\G$ is $k$-regular and periodic,
then we have $\frac{16t}{k^3} \in \MB{Z}$.
}
\end{pro}
\begin{proof}
We have
\begin{align*}
\alpha_{2n-3} &= 8d_{n-3} \tag{by Corollary~\ref{0415-02}~(iv)} \\
&= 8 \cdot \frac{c_{n-3}}{k^3} \tag{by Lemma~\ref{0415-03}} \\
&= -\frac{16t}{k^3} \tag{by Lemma~\ref{0426-01}}.
\end{align*}
Proposition~\ref{0416-2}~(iii) implies $\alpha_{2n-3} \in \MB{Z}$.
Therefore, we have the statement.
\end{proof}

\subsection{Periodic graphs with a prime number of vertices}

The condition in Corollary~\ref{0416-3} leads to an interesting fact
when the number of vertices in a regular graph is prime.

\begin{thm}
{\it
Let $\eta \in \MB{Q}\pi$,
and let $G$ be a $k$-regular weakly connected mixed graph with $p$ vertices,
where $p$ is a prime number.
Then the mixed graph $G$ is periodic if and only if $G^{\pm}$ is isomorphic to either
the cycle graph $C_p$ or the path graph $P_2$.
}
\end{thm}

\begin{proof}
Suppose $G$ is periodic.
We have $\frac{2p}{k} \in \MB{Z}$ by Corollary~\ref{0416-3}.
Since $p$ is prime and $k < p$, we have $k = 1,2$.
The mixed graph $G$ is weakly connected,
so $G^{\pm}$ is isomorphic to either $C_p$ or $P_2$.
Conversely,
mixed graphs whose underlying graphs are isomorphic to the cycle graphs or the path graphs
are periodic since $\eta \in \MB{Q}\pi$.
This is already shown in \cite{KSeYa}.
\end{proof}

\subsection{Examples}

In previous works,
several periodic regular undirected graphs in the sense of conventional Grover walks have been found.
We list the number of vertices, valency, and the number of triangles in those graphs.
See Table~\ref{0505-2}.
We can see that the graphs satisfy $\frac{2n}{k}, \frac{16t}{k^3} \in \MB{Z}$.

\begin{table}[h]
  \centering
  \begin{tabular}{|c|c|c|c|c|}
\hline
Graphs & $n$ & $k$ & $t$ & Ref. \\
\hline
\hline
$C_n$ & $n$ & $2$ & $0, 1$ & well-known  \\ \hline
$K_{k,k}$ & $2k$ & $k$ & $0$ & \cite{HKSS2017} \\ \hline
$K_{\lambda, \lambda, \lambda}$ & $3\lambda$ & $2\lambda$ & $\lambda^3$ & \cite{HKSS2017} \\ \hline
$H(3,3)$ & $27$ & $6$ & $27$ & \cite{Y2019}  \\ \hline
$H(4,2)$ & $16$ & $4$ & $0$ & \cite{Y2019} \\ \hline
\end{tabular}
 \caption{Periodic regular graphs known so far} \label{0505-2}
\end{table}

\section{Periodicity of mixed complete graphs} \label{0504-5}
We say that a mixed graph $G$ is a {\it mixed complete graph}
if $G^{\pm}$ is isomorphic to the complete graph.
In this section, we discuss the periodicity of mixed complete graphs.
In the case of undirected graphs,
there is only one complete graph for $n$.
On the other hand,
there are $3^{\frac{n(n-1)}{2}}$ distinct mixed complete graphs on any vertex set $V$ of size $n$.
Moreover, there are infinitely many ways to take $\eta$.
However, the mixed complete graphs are not periodic for any $n \geq 4$ and $\eta$.
We prove this fact by the previous arguments.

\begin{thm}
{\it
Let $G$ be a mixed complete graph equipped with an $\eta$-function $\theta$.
If $n \geq 4$, then $G$ is not periodic.
}
\end{thm}

\begin{proof}
Assume that $G$ is periodic.
Since $G$ is $(n-1)$-regular, we have $\frac{2n}{n-1} \in \MB{Z}$ by Corollary~\ref{0416-3}.
However, $\frac{2n}{n-1}$ is not integer for $n \geq 4$.
We have a contradiction.
\end{proof}

We supplement the case of $n = 2, 3$.
When $n = 2$, a mixed $K_2$ is a mixed path.
In \cite{KSeYa}, the authors show that the mixed paths are periodic for any $\eta$.
When $n = 3$, a mixed $K_3$ is a mixed cycle.
The authors in \cite{KSeYa} show that a mixed cycle is periodic if and only if $\eta \in \MB{Q}\pi$.

\section{Mixed strongly regular graphs} \label{0902-6}
We say that a mixed graph $G$ is a {\it mixed strongly regular graph with parameters $(n, k ,\lambda, \mu)$} if $G^{\pm}$ is a strongly regular graph with parameters $(n, k ,\lambda, \mu)$.
See \cite{BH, GR} for the definition and basis properties of strongly regular graphs.
Periodic undirected strongly regular graphs have been fully determined by Higuchi et al \cite{HKSS2017}.
In this section,
we consider the mixed graph analogy and an alternative proof of their result.
Before starting discussion, we quote a well-known fact.

\begin{lem}[\cite{BH}, Theorem~9.1.3 (iv)] \label{0805-2}
{\it
Let $\G$ be a strongly regular graph with parameters $(n, k ,\lambda, \mu)$.
If $\G$ has a non-integral eigenvalue,
then $(n, k ,\lambda, \mu) = (4t+1, 2t, t-1, t)$ for some integer $t$.
In particular, $\lambda - \mu = -1$.
}
\end{lem}


\begin{thm} \label{0803-1}
{\it
Let $G$ be a weakly connected mixed strongly regular graph with parameters $(n, k ,\lambda, \mu)$.
If $G$ is periodic,
then we have $(n, k ,\lambda, \mu) = (3\lambda, 2\lambda, \lambda, 2\lambda)$,
$(2k, k, 0, k)$, $(5,2,0,1)$ or  $(15, 6, 1, 3)$.
}
\end{thm}

\begin{proof}
We recall that $k(k-\lambda-1) = (n-k-1)\mu$.
Since $G$ is weakly connected, we have $\mu > 0$.
Thus
\begin{equation} \label{0803-2}
n = \frac{k^2 + (\mu-\lambda -1)k + \mu}{\mu}
\end{equation}
By Corollary~\ref{0416-3},
we have $\frac{2n}{k} = \frac{2}{\mu}(k+\mu-\lambda-1 + \frac{\mu}{k}) \in \MB{Z}$.
In particular, $2(k+\mu-\lambda-1 + \frac{\mu}{k}) \in \mu\MB{Z} \subset \MB{Z}$,
so $\frac{2\mu}{k} \in \MB{Z}$.
Since $0 < \mu \leq k$,
we have $0 < \frac{2\mu}{k} \leq 2$, i.e., $\frac{2\mu}{k} = 1, 2$.

We first consider the case $\frac{2\mu}{k} = 2$.
By (\ref{0803-2}), we have
\begin{equation} \label{0803-4} 
n - 2k + \lambda = 0.
\end{equation}
Then $\bar{\mu} = n - 2k + \lambda = 0$ is obtained
for the parameters $(n, \bar{k}, \bar{\lambda}, \bar{\mu})$ of the complement $\overline{G^{\pm}}$.
By Lemma~10.1 of \cite{GR},
there exists $m > 1$ such that $\overline{G^{\pm}} = mK_{\bar{k}+1} = mK_{n-k}$,
where $mK_a$ is the disjoint union of the $m$ complete graphs on $a$ vertices.
Taking the complements of both sides and comparing the number of vertices and valency,
we have
\begin{equation} \label{0803-3}
n=m(n-k)
\end{equation}
and $k=(m-1)(n-k)$.
From Corollary~\ref{0416-3} again,
we obtain $\frac{2n}{k} = 2 + \frac{2}{m-1} \in \MB{Z}$, that is, $m=2, 3$.
If $m=2$, then $n = 2k$ from (\ref{0803-3}), and $G^{\pm} = \overline{2K_{k}} = K_{k,k}$.
The parameters are $(2k, k, 0, k)$.
If $m=3$, then $n = \frac{3}{2}k$ from (\ref{0803-3}), and $k = 2\lambda$ from (\ref{0803-4}).
Thus we have $(n, k ,\lambda, \mu) = (3\lambda, 2\lambda, \lambda, 2\lambda)$.
Note that this graph is $G^{\pm} = \overline{3K_{\lambda}} = K_{\lambda, \lambda, \lambda}$.

Next, we consider the case $\frac{2\mu}{k}=1$.
Then (\ref{0803-2}) becomes
\begin{equation} \label{0805-1}
n = 3k-2\lambda-1.
\end{equation}
By Corollary~\ref{0416-3},
we have $\frac{2n}{k} = 6 - \frac{4\lambda+2}{k} \in \MB{Z}$,
i.e., $\frac{4\lambda+2}{k} \in \MB{Z}$.
On the other hand, since $0 \leq \lambda \leq k-1$,
we have $0 < \frac{2}{k} \leq \frac{4\lambda+2}{k} \leq 4 - \frac{2}{k} < 4$.
This implies $\frac{4\lambda+2}{k} = 1, 2, 3$.
Suppose $\frac{4\lambda+2}{k} = 1$.
Then (\ref{0805-1}) becomes $n = 10\lambda + 5$.
In this case,
we have $(n, k ,\lambda, \mu) = (10\lambda+5, 4\lambda+2, \lambda, 2\lambda+1)$.
We show that such undirected strongly regular graphs exist only if $\lambda = 0,1$.
The eigenvalues that are not $k$ are
\begin{equation} \label{0805-3}
\frac{-\lambda - 1 \pm \sqrt{\lambda^2 + 10\lambda + 5}}{2}.
\end{equation}
We first find the condition for $\lambda$ where these eigenvalues are integer.
For these eigenvalues to be integer,
it must be hold that $\lambda^2 + 10\lambda + 5 = s^2$ for some non-negative integer $s$.
This equality can be transformed to $(\lambda + s + 5)(\lambda - s + 5) = 20$,
which implies $\lambda = 1$ and $s = 4$.
We have $(n, k ,\lambda, \mu) = (15, 6, 1, 3)$.
We next consider the eigenvalues (\ref{0805-3}) are not integer.
In this case, we have $\lambda - (2\lambda+1) = -1$ by Lemma~\ref{0805-2}.
Thus we have $\lambda = 0$.
We have $(n, k ,\lambda, \mu) = (5, 2, 0, 1)$.
If $\frac{4\lambda+2}{k} = 2$, then $k = 2\lambda+1$, which is odd,
but in this case $\mu = \frac{k}{2} \not\in \MB{Z}$.
Suppose $\frac{4\lambda+2}{k} = 3$.
Then (\ref{0805-1}) becomes $n = 2\lambda + 1$.
In this case, we have
$(n, k ,\lambda, \mu) = (2\lambda+1, \frac{2}{3}(2\lambda+1), \lambda, \frac{1}{3}(2\lambda + 1))$.
We show that no undirected strongly regular graphs with such parameters exist.
The eigenvalues that are not $k$ are
\begin{equation} \label{0806-1}
\frac{\lambda-1 \pm \sqrt{\lambda^2+22\lambda+13}}{6}.
\end{equation}
We first find the condition for $\lambda$ where these eigenvalues are integer.
For these eigenvalues to be integer,
it must be hold that $\lambda^2 + 22\lambda + 13 = s^2$ for some non-negative integer $s$.
This equality can be transformed to $(\lambda + s + 11)(\lambda - s + 11) = 108$,
which implies $(\lambda, s) = (1, 6), (17, 26)$.
For $\lambda = 1$ this graph is a complete graph,
and for $\lambda = 17$ the valency $k = \frac{70}{3} \not\in \MB{Z}$.
Both cases are irrational.
We next consider the eigenvalues (\ref{0806-1}) are not integer.
In this case, we have $\lambda - \frac{1}{3}(2\lambda + 1) = -1$ by Lemma~\ref{0805-2}.
However, the solution $\lambda = -2$ is negative.
Consequently,
strongly regular graphs with parameters $(n, k ,\lambda, \mu) = (2\lambda+1, \frac{2}{3}(2\lambda+1), \lambda, \frac{1}{3}(2\lambda + 1))$ do not exist.
\end{proof}

Comparing with the result of Higuchi et al \cite{HKSS2017},
the parameters $(15, 6, 1, 3)$ newly appear.
This undirected strongly regular graph is uniquely determined from the parameters (see Lemma~10.9.3 of \cite{GR}), 
and is the complement of the line graph of the complete graph on 6 vertices,
i.e., it is $\overline{L(K_6)}$.

As for undirected periodic strongly regular graphs,
it is possible to eliminate the parameters $(15, 6, 1, 3)$ by using Proposition~\ref{0806-2}.
We make a slightly general argument to count the number of triangles from the parameters.
A $k$-regular undirected graph on $n$ vertices
is said to be {\it edge-regular with parameters $(n,k,\lambda)$}
if any two adjacent vertices have precisely $\lambda$ common neighbors.

\begin{lem}[Page~3 of \cite{BCN}] \label{0811-3}
{\it
Let $\G = (V, E)$ be an edge-regular graph with parameters $(n, k, \lambda)$.
Then, the number of triangles in $\G$ is $\frac{nk\lambda}{6}$.
}
\end{lem}

\begin{proof}
Let $A$ be the adjacency matrix of $\G$,
and let $t$ be the number of triangles in $\G$.
We consider the set of closed walks $\MC{T}_x$ from $x$ to $x$ of length $3$,
i.e., $\MC{T}_x = \{ (y,z) \in V \times V \mid x \sim y \sim z \sim x \}$,
where $u \sim v$ denotes that $u$ and $v$ are adjacent.
Then we have
\[
|\MC{T}_x| = \sum_{ \substack{y \in V \\ y \sim x}} | \{ z \in V \mid y \sim z \sim x \} |
= \sum_{ \substack{y \in V \\ y \sim x}} \lambda
= k\lambda.
\]
Thus,
$6t = \Tr(A^3) = \sum_{x \in V}(A^3)_{x,x} = \sum_{x \in V}|\MC{T}_x| = nk\lambda$,
which implies the statement.
\end{proof}

By Lemma~\ref{0811-3},
the number of triangles $t$ in the $(15, 6, 1, 3)$-strongly regular graph is $15$.
On the other hand, $\frac{16t}{k^3} = \frac{10}{9} \not\in \MB{Z}$,
and hence this strongly regular graph is not periodic by Proposition~\ref{0806-2}.
With this discussion,
the result of Higuchi et al \cite{HKSS2017} was restored in the sense of the necessary condition.
Note that the consideration made here is only necessary conditions for the periodicity.
A simple way to claim
that undirected strongly regular graphs that pass the necessary condition is indeed periodic
is to explicitly determine their eigenvalues.

\section{Mixed Hamming graphs} \label{0902-7}

We say that a mixed graph $G$ is a {\it mixed Hamming graph $H(d,q)$}
if $G^{\pm}$ is isomorphic to the Hamming graph $H(d,q)$.
We omit the definition of Hamming graphs.
Readers are referred to \cite{BCN}.
Periodic undirected Hamming graphs have been fully determined by Yoshie \cite{Y2019}.
In this section, we study its mixed graph analogy.

\begin{lem} \label{0812-1}
{\it
Let $x$ and $y$ be positive integers that are coprime.
For any $r \in \MB{Q}$ such that $x^r$ is integer,
$x^r$ and $y$ are coprime.
}
\end{lem}

\begin{proof}
Let the prime factorizations of $x$ and $y$ be
$x = p_1^{e_1}p_2^{e_2} \cdots p_t^{e_t}$ and $y = q_1^{f_1}q_2^{f_2} \cdots q_t^{f_t}$, respectively.
Since $x$ and $y$ are coprime, 
the prime numbers $p_1, \dots, p_s, q_1, \cdots, q_t$ are all distinct.
Then we have $x^r = p_1^{re_1}p_2^{re_2} \cdots p_t^{re_t}$,
but $p_1, \dots, p_s, q_1, \cdots, q_t$ are all distinct, so $x^r$ and $y$ are coprime.
\end{proof}

\begin{pro}
{\it
Let $G$ be a mixed Hamming graph $H(d,q)$.
If $G$ is periodic,
then $(d, q) = (2^f, 2)$ or $(3^f, 3)$ for some non-negative integer $f$.
}
\end{pro}

\begin{proof}
Note that the Hamming graph $H(d,q)$ has $q^d$ vertices and its valency is $d(q-1)$.
By Corollary~\ref{0416-3}, we have $\frac{2q^d}{d(q-1)} \in \MB{Z}$,
so $\frac{2q^d}{q-1} \in d\MB{Z} \subset \MB{Z}$.
Since $q$ and $q-1$ are coprime,
Lemma~\ref{0812-1} derives that $q^d$ and $q-1$ are coprime.
This implies $\frac{2}{q-1} \in \MB{Z}$, and hence $q=2,3$.
If $q=2$, then by Corollary~\ref{0416-3} again,
we have $\frac{2^{d+1}}{d} \in \MB{Z}$.
Thus, we have $d = 2^f$ for some $f$.
The same applies to the case $q=3$.
\end{proof}

We supplement on the undirected Hamming graphs.
In Yoshie's work \cite{Y2019},
we know that the undirected Hamming graphs $H(d,q)$ are periodic
only in the case $(d,q) = (1,2)$, $(1,3)$, $(2,2)$, $(2,3)$, or $(4,2)$.
However, even if the number of triangles is counted using Lemma~\ref{0811-3}
in both cases $q=2$ and $3$,
the periodicity can not be eliminated by Proposition~\ref{0806-2}.

\section{Mixed sesquilinear forms graphs} \label{0902-8}

In this section,
we provide good examples where the periodicity is eliminated using only Corollary~\ref{0416-3}.
As in the previous sections,
we say that a mixed graph $G$ is a {\it mixed $\G$} for an undirected graph $\G$
if $G^{\pm}$ is isomorphic to $\G$.
The examples provided here are sesquilinear forms graphs, whose details are in Chapter~9 of \cite{BCN}.
Definition, the number of vertices and valency of each graph are summarized in Table~\ref{0813-1}.

\begin{table}[h]
  \centering
  \begin{tabular}{|c|c|c|c|c|}
\hline
Notation & Vertex set & Adjacency & \#Vertices & Valency \\
\hline
\hline 
${\rm Bil}_{d \times e}(q)$ & $\MB{F}_{q}^{d \times q}$ & $\rank(M -N) = 1$ & $q^{de}$ & $\frac{(q^d -1)(q^e-1)}{q-1}$ \\ \hline
${\rm Alt}_{n}(q)$ & $\{ M \in \MB{F}_{q}^{n \times n} \mid M^{\top} = -M\}$ & $\rank(M -N) = 2$ & $q^{\frac{n(n-1)}{2}}$ & $\frac{(q^n -1)(q^{n-1}-1)}{q^2 - 1}$ \\ \hline
${\rm Her}_{n}(q^2)$ & $\{ M \in \MB{F}_{q^2}^{n \times n} \mid M^* = M\}$ & $\rank(M -N) = 1$ & $q^{n^2}$ & $\frac{q^{2n} - 1}{q+1}$ \\ \hline
${\rm Sym}_{n}(q)$ & $\{ M \in \MB{F}_{q}^{n \times n} \mid M^{\top} = M\}$ & $\rank(M -N) = 1$ & $q^{\frac{n(n+1)}{2}}$ & $q^n - 1$ \\ \hline
\end{tabular}
 \caption{Sesquilinear forms graphs} \label{0813-1}
\end{table}

Note that bilinear forms graphs ${\rm Bil}_{d \times e}(q)$, alternating forms graphs ${\rm Alt}_{n}(q)$
and Hermitian forms graphs ${\rm Her}_{n}(q^2)$ are distance-regular graphs,
but symmetric bilinear forms graphs ${\rm Sym}_{n}(q)$ are not distance-regular in general.

\begin{thm}
{\it
None of mixed ${\rm Bil}_{d \times e}(q)$, mixed ${\rm Alt}_{n}(q)$ for $n \geq 4$,
mixed ${\rm Her}_{n}(q^2)$ and mixed ${\rm Sym}_{n}(q)$ are periodic.
}
\end{thm}

\begin{proof}
We show only for mixed ${\rm Bil}_{d \times e}(q)$ and mixed ${\rm Alt}_{n}(q)$.
The same is shown for the remaining graphs.

First, we show for mixed ${\rm Bil}_{d \times e}(q)$. 
We assume $e \geq d \geq 2$ without loss of generality,
and suppose that ${\rm Bil}_{d \times e}(q)$ is periodic.
By Corollary~\ref{0416-3}, we have
\begin{equation} \label{0815-1}
\frac{2(q-1)q^{de}}{(q^d - 1)(q^e - 1)} \in \MB{Z}.
\end{equation}
Since $q^d$ and $q^d-1$ are coprime, $q^{de}$ and $q^d - 1$ are also coprime by Lemma~\ref{0812-1}.
Similarly, we see that $q^{de}$ and $q^e - 1$ are coprime.
Thus, (\ref{0815-1}) holds if and only if
\begin{equation} \label{0815-2}
\frac{2(q-1)}{(q^d - 1)(q^e - 1)} \in \MB{Z}.
\end{equation}
On the other hand, since $e \geq d \geq 2$ and $q \geq 2$,
we have
\[ 0 < \frac{2(q-1)}{(q^d - 1)(q^e - 1)} < \frac{2(q-1)}{(q^1 - 1)(2^2 - 1)} = \frac{2}{3} < 1. \]
This is a contradiction.

Next, we show for mixed ${\rm Alt}_{n}(q)$.
Suppose that ${\rm Alt}_{n}(q)$ is periodic.
By Corollary~\ref{0416-3}, we have
\begin{equation} \label{0815-3}
\frac{2q^{\frac{n(n-1)}{2}}(q^2 - 1)}{(q^n - 1)(q^{n-1} - 1)} \in \MB{Z}.
\end{equation}
Since $q^n$ and $q^n-1$ are coprime, $q^{n \cdot \frac{n-1}{2}}$ and $q^n - 1$ are also coprime by Lemma~\ref{0812-1}.
Similarly, we see that $q^{(n-1) \cdot \frac{n}{2}}$ and $q^{n-1} - 1$ are coprime.
Thus, (\ref{0815-3}) holds if and only if
\begin{equation} \label{0815-2}
\frac{2(q^2-1)}{(q^n - 1)(q^{n-1} - 1)} \in \MB{Z}.
\end{equation}
On the other hand, since $n > 3$ and $q \geq 2$,
we have
\[ 0 < \frac{2(q^2-1)}{(q^n - 1)(q^{n-1} - 1)} < \frac{2(q^2-1)}{(q^2 - 1)(2^{3-1} - 1)} = \frac{2}{7} < 1. \]
This is a contradiction.
\end{proof}

Note that both ${\rm Alt}_{n}(q)$ for $n=2, 3$ form the complete graphs.
The periodicity of mixed complete graphs is as discussed in Section~\ref{0504-5}.

\section{Mixed Johnson graphs} \label{0902-9}

In this section, we study the periodicity of mixed $J(n,k)$.
See \cite{BCN} for the definition and properties of Johnson graphs.
In Yoshie's work \cite{Y2019},
it is shown that undirected Johnson graphs $J(n,k)$ are periodic
if and only if $(n,k) = (2,1), (3,1)$ or $(4,2)$.
As we will see later,
the mixed graph analogy of his work is currently not as successful as in the previous sections.
However, we can provide an upper bound on the number of vertices when mixed $J(n,k)$ is periodic.
Recall that the Johnson graph $J(n,k)$ is a $k(n-k)$-regular graph with $\binom{n}{k}$ vertices.

\begin{pro}
{\it
If a mixed $J(n,k)$ is periodic,
then $n \leq k! + k$.
}
\end{pro}

\begin{proof}
We consider the condition that $\frac{2}{k(n-k)}\binom{n}{k}$ is an integer.
Let $N = n-k$,
and let $f(N) = (N+k)(N+k-1) \cdots (N+1) - k!$.
Then, we have $\frac{f(N)}{N} \in \MB{Z}$.
By Corollary~\ref{0416-3},
\begin{align*}
\MB{Z} \ni \frac{2}{k(n-k)}\binom{n}{k}
&= \frac{1}{k \cdot \frac{k!}{2} (n-k)} \cdot n(n-1) \cdots (n-k+1) \\
&= \frac{1}{k \cdot \frac{k!}{2}} \cdot \frac{1}{N} \cdot (N+k)(N+k-1) \cdots (N+1) \\
&= \frac{1}{k \cdot \frac{k!}{2}} \cdot \frac{1}{N} \cdot (f(N) + k!) \\
&= \frac{1}{k \cdot \frac{k!}{2}} \left( \frac{f(N)}{N} + \frac{k!}{N} \right).
\end{align*}
Thus,
$\frac{f(N)}{N} + \frac{k!}{N} \in \left( k \cdot \frac{k!}{2} \right) \MB{Z} \subset \MB{Z}$.
Since $\frac{f(N)}{N} \in \MB{Z}$, we have $\frac{k!}{N} = \frac{k!}{n-k} \in \MB{Z}$.
In particular, $\frac{k!}{n-k} \geq 1$, so the claim is obtained.
\end{proof}

From this proposition,
for each $k$, there are finite number of $n$'s for which $J(n,k)$ can be periodic.
By computer,
we find that $\frac{2}{k(n-k)}\binom{n}{k}$ is not integer for $3 \leq k \leq 11$.
Thus, mixed $J(n,k)$ is not periodic for such $k$.
However, when $k=12$ and $n=27$,
$\frac{2}{k(n-k)}\binom{n}{k} = \frac{2}{12 \cdot 15}\binom{27}{12} = 193154$ is integer.

\section*{Acknowledgements}
The author would like to thank the reviewer
for valuable comments and suggestions for improving the manuscript.
The author is supported by JSPS KAKENHI (Grant No. 20J01175).

\end{document}